\documentclass[12pt]{article}
\usepackage{amssymb,amsmath,amsfonts,amsthm}

\newcounter{parag}

\newtheorem{lem}{Lemma}
\newtheorem*{theorem}{Theorem}

\newtheorem*{cor}{Corollary}
\newtheorem{prop}{Proposition}

\newcommand{\Aut}{\operatorname{Aut}}

\DeclareMathOperator{\X}{\mathfrak{X}}
\DeclareMathOperator{\Y}{\mathfrak{Y}}

\begin{document}

$\mbox{   }$ $\mbox{   }$$\mbox{   }$$\mbox{   }$$\mbox{   }$$\mbox{   }$$\mbox{   }$$\mbox{   }$$\mbox{   }$$\mbox{   }$$\mbox{   }$$\mbox{   }$$\mbox{   }$$\mbox{   }$$\mbox{   }$$\mbox{   }$$\mbox{   }$$\mbox{   }$$\mbox{   }$$\mbox{   }$$\mbox{   }$$\mbox{   }$$\mbox{   }$$\mbox{   }$$\mbox{   }$$\mbox{   }$$\mbox{   }$$\mbox{   }$$\mbox{   }$$\mbox{   }$$\mbox{   }$$\mbox{   }$$\mbox{   }$$\mbox{   }$$\mbox{   }$$\mbox{   }$$\mbox{   }$$\mbox{   }$$\mbox{   }$$\mbox{   }$$\mbox{   }$$\mbox{   }$$\mbox{   }$$\mbox{   }$$\mbox{   }$$\mbox{   }$$\mbox{   }$$\mbox{   }$$\mbox{   }$$\mbox{   }$$\mbox{   }$$\mbox{   }$$\mbox{   }$$\mbox{   }$$\mbox{   }$$\mbox{   }$$\mbox{   }$$\mbox{   }$$\mbox{   }$$\mbox{   }$$\mbox{   }$$\mbox{   }$$\mbox{   }$$\mbox{   }$$\mbox{   }$$\mbox{   }$ MSC 20D06, 20D60, 20E07, 20D20

\medskip
\medskip
\medskip

\begin{center}
{\bf {\large Finite simple exceptional groups of Lie type in which all the subgroups of odd index are pronormal}}

\medskip
\medskip
\medskip

\medskip
\medskip
\medskip

A.~S.~Kondrat'ev, N.~V.~Maslova, D.~O.~Revin

\end{center}

\bigskip

{\bf Abstract} {\it  A subgroup $H$ of a group $G$ is said to be {pronormal} in $G$ if $H$ and $H^g$ are conjugate in $\langle H, H^g \rangle$ for every $g \in G$. In this paper, we classify finite simple groups $E_6(q)$ and ${}^2E_6(q)$ in which all the subgroups of odd index are pronormal. Thus, we complete a classification of finite simple exceptional groups of Lie type in which all the subgroups of odd index are pronormal.

\smallskip

\smallskip
}
\bigskip

\section{Introduction}

Throughout the paper we consider only finite groups, and henceforth the term group means finite group.

According to Ph. Hall \cite{Hall_lectures}, a subgroup $H$ of a group $G$ is said to be {\it pronormal} in $G$ 
if $H$ and $H^g$ are conjugate in $\langle H, H^g \rangle$ for every $g \in G$.

Some of well-known examples of pronormal subgroups are the following: normal subgroups, maximal subgroups, Sylow subgroups, Sylow subgroups of proper normal subgroups, Hall subgroups of solvable groups.

In 2012, E.~Vdovin and the third author \cite{VR} proved that the Hall subgroups (when they exist) are pronormal in all simple groups and, guided by the analysis in their proof, they conjectured that any subgroup of odd index of a simple group is pronormal in this group. This conjecture was disproved in \cite{KMR_2,KMR_3}. Precisely, if $q \equiv \pm 3\pmod 8$ and $n \not \in \{2^m, 2^m(2^{2k}+1) \mid m, k \in \mathbb{N} \cup \{0\}\}$, then the simple symplectic group $PSp_{2n}(q)$ contains a non-pronormal subgroup of odd index. Thus, the problem of classification of simple nonabelian groups in which the subgroups of odd index are pronormal naturally arises.

In \cite{KMR_1}, we confirmed the conjecture for many families of simple groups. Namely, it was proved that the subgroups of odd index are pronormal in the following simple groups{\rm:}
$A_n$, where $n\ge 5$,
the sporadic groups,
the groups of Lie type over fields of characteristic $2$,
$PSL_{2^n}(q)$,
$PSU_{2^n}(q)$,
$PSp_{2n}(q)$, where  $q\not\equiv\pm3 \pmod 8$,
$P\Omega_{2n+1}(q)$, $P\Omega_{2n}^\pm(q)$,
the exceptional groups of Lie type not isomorphic to $E_6(q)$ or ${}^2E_6(q)$.

Moreover, in \cite{KMR_3,KMR_4} we have proved that if $q \equiv \pm 3\pmod 8$ and $n \in \{2^m, 2^m(2^{2k}+1) \mid m, k \in \mathbb{N} \cup \{0\}\}$, then all the subgroups of odd index are pronormal in the simple symplectic group $PSp_{2n}(q)$. So, we obtained the complete classification of simple symplectic groups in which all the subgroups of odd index are pronormal.

We use the following notation: $E_6(q)=E^+_6(q)$ and ${}^2E_6(q)=E_6^-(q)$.

In this paper, we prove the following theorem.

\begin{theorem} Let $G=E^{\varepsilon}_6(q)$, where $\varepsilon\in\{+,-\}$, $q=p^m$, and $p$ is a prime.
All the subgroups of odd index are pronormal in $G$ if and only if  $q\not\equiv \varepsilon1\pmod {18}$ and if $\varepsilon= +$, then $m$ is a power of $2$.

\end{theorem}

So, we complete a classification of simple exceptional groups of Lie type in which all the subgroups of odd index are pronormal and obtain the following result.

\begin{cor} Let $G$ be a simple exceptional group of Lie type in characteristic $p$.
Then $G$ contains a non-pronormal subgroup of odd index if and only if $G\cong E^{\varepsilon}_6(p^m)$, where $\varepsilon\in\{+,-\}$ and $q\equiv \varepsilon1\pmod {18}$
or $\varepsilon= +$ and $m$ is not a power of $2$.

\end{cor}

The problem of classification of simple nonabelian groups in which all the subgroups of odd index are pronormal is still open for the following simple groups: $PSL_{n}(q)$ and $PSU_{n}(q)$, where $q$ is odd and $n$ is not a power of $2$.

 A more detailed survey of investigations on pronormality of subgroups of odd index in finite (not necessary simple) groups can be found in survey papers~\cite{GR_Surv,KMR_Surv}. These surveys contain new results and some conjectures and open problems. One such open problem involves the classification of direct products of nonabelian simple groups in which the subgroups of odd index are pronormal.

\section{ Preliminaries}

Our terminology and notation are mostly standard and could be found, for example, in \cite{Car, Atlas, G, GLS, KL}.

Given a set $\pi$ of primes, $\pi'$ stands for the set of all primes not in $\pi$. Also, if $n$ is a positive integer,
then $n_\pi$ is the largest natural divisor of $n$ such that all the prime divisors of $n_\pi$ are in $\pi$.

For a group $G$ and a subset $\pi$ of the set of all primes, $O_\pi(G)$ and $Z(G)$ denote the $\pi$-radical (the largest normal $\pi$-subgroup) and the center of $G$. Also we will write $O(G)$ instead of $O_{2'}(G)$.

We denote by $E(G)$ the {\it layer} of $G$, i.\,e. the subgroup of $G$ generated by all components (subnormal quasisimple subgroups) of $G$.

The symmetric group of degree $n$ is denoted by $Sym_n$.

Let $P$ be a $p$-group, where $p$ is a prime, and let $n$ be a positive integer. Following \cite[page~17]{G}, we define

$$\Omega_n(P)=\langle  x\in P \mid |x|\le p^n\rangle\text{ and } \mho^n(P)=\langle x^{p^n}\mid x\in P\rangle. $$

The {\it rank} of an abelian group $G$ is the least number of generators of $G$. It is well-known that the rank of a subgroup of an abelian group $A$ is at most the rank of $A$.

\begin{lem}\label{p'AutOfAbelianp}  {\rm \cite[Chapter~5, Theorems~2.3 and~2.4]{G}}
Let $P$ be an abelian $p$-group, where $p$ is a prime, and let $A$ be a $p'$-subgroup of $\Aut(P)$. Then
$$
P=C_P(A)\times [A,P],
$$
and if $A$ acts trivially on $\Omega_1(P)$, then $A=1$.
\end{lem}

\begin{lem}{\rm \cite[Lemma~5]{VR}} \label{NormSyl} Suppose that $G$ is a group and $H \le G$. Assume also that $H$ contains a Sylow subgroup $S$ of $G$. Then $H$ is pronormal in $G$ if and only if the subgroups $H$ and $H^g$ are conjugate in $\langle H, H^g \rangle$ for every $g \in N_G(S)$.

\end{lem}

\begin{lem}{\rm \cite[Lemma~5]{KMR_1}} \label{Overgroup} Suppose that $H$ and $M$ are subgroups of a group  $G$ and $H \le M$. Then

$(1)$ if $H$ is pronormal in $G$, then $H$ is pronormal in  $M${\rm;}

$(2)$ if $S \le H$ for some Sylow subgroup $S$ of $G$, $N_G(S) \le M$, and $H$ is pronormal in $M$, then $H$ is pronormal in $G$.

\end{lem}








We use the following notation. Fix a prime $p$. Let $\X_p$ be the class of all the groups in which a Sylow $p$-supgroup is self-normalized, and let $\Y_p$ be the class of all the groups in which all the subgroups, whose indices are coprime to $p$, are pronormal.

\begin{lem}\label{N_G(S)=S} {\rm \cite[Lemma~2]{KMR_3}} The following statements hold{\rm:}

$(1)$ If $X\unlhd Y$, $X\in \X_p$, and $Y/X\in \X_p$, then $Y\in \X_p$.

$(2)$ If $X\le Y\in \X_p$ and the index $|Y:X|$ is coprime to $p$, then $X$ is pronormal in $Y$.

\end{lem}

\begin{lem}\label{HV=>VB}{\rm \cite[Lemma~6]{GMR}}
Let $N$ be a normal subgroup of a group $G$ such that $G/N\in\X_p$, and let $H$ be a subgroup of $G$ whose index in $G$ is coprime to $p$. Then $H$ is pronormal in $G$
if and only if $H$ is pronormal in $HN$.
\end{lem}

\begin{lem}\label{Important} {\rm \cite[Theorem~1]{GMR}}
Let $G$ be a group and $A$ be a normal subgroup of $G$ such that $A\in\Y_p$ and $G/A\in \X_p$.
Let $T$ be a Sylow $p$-subgroup of $A$. Then  $G\in \Y_p$ if and only if $N_G(T)/T \in \Y_p$.

\end{lem}

\begin{lem}{\rm \cite[Theorem~1]{KMR_2}} \label{CrterionOfProSub}
Let $H$ and $V$ be subgroups of a group $G$ such that
$V$ is an abelian normal subgroup of $G$ and
 $G=HV$. Then the following statements are equivalent{\rm:}

$(1)$ $H$ is pronormal in $G${\rm;}

$(2)$  $U=N_U(H)[H,U]$ for each $H$-invariant subgroup $U$ of $V$.

\end{lem}

To prove the main result of this paper we need two following easy-proving assertions on finite fields.

\begin{lem}\label{Fields1} Let $q$ be a power of an odd prime. Let $\alpha$ be a generator of a Sylow $2$-subgroup of the multiplicative group $\mathbb{F}_{q^2}^*$ of the field  $\mathbb{F}_{q^2}$. Then the elements
$1,\alpha$ form a basis of $\mathbb{F}_{q^2}$ as a vector space over its subfield~$\mathbb{F}_q$.

\end{lem}

\begin{proof} Note that $\dim_{\mathbb{F}_{q}}^{}\mathbb{F}_{q^2}=2$. Thus, it is sufficient to prove that $\alpha\notin \mathbb{F}_q$. We have
$$|\mathbb{F}_{q^2}^*|^{}_2=(q+1)_2(q-1)_2\ge 2(q-1)_2>(q-1)_2=|\mathbb{F}_{q}^*|^{}_2.$$ So, a Sylow $2$-subgroup of  $\mathbb{F}_{q^2}^*$ is not contained in  $\mathbb{F}_q$.
\end{proof}

\begin{lem}\label{Fields2}
Let $p$ be an odd prime and $q=p^{2^k}$. Assume that $\alpha$ is a generator of a Sylow $2$-subgroup of the multiplicative group $\mathbb{F}_q^*$ of the field  $\mathbb{F}_q$. Then elements
$1,\alpha,\alpha^2,\dots, \alpha^{2^k-1}$ form a basis of $\mathbb{F}_q$ as a vector space over its subfield~$\mathbb{F}_p$.

\end{lem}

\begin{proof} We use inductive reasonings with respect to $k$.

The case $k=0$ is trivial, and in the case $k=1$ we refer to Lemma~\ref{Fields1}.

Suppose that $k>1$ and $q_0=p^{2^{k-1}}$. Note that $q_0$ is a square, therefore, $q_0+1 \equiv 2 \pmod{4}$. We have $$|\mathbb{F}_{q}^*|^{}_2=|\mathbb{F}_{q_0^2}^*|^{}_2=(q_0+1)_2(q_0-1)_2= 2(q_0-1)_2=2|\mathbb{F}_{q_0^{}}^*|^{}_2.$$
Therefore, $\alpha^2\in \mathbb{F}_{q_0^{}}^*$. Moreover, $\beta=\alpha^2$ is a generator of the Sylow $2$-subgroup of  $\mathbb{F}_{q_0^{}}^*$.

Now using the inductive hypothesis we prove that the elements $$1,\beta,\beta^2,\dots, \beta^{2^{k-1}-1}$$ form a basis of $\mathbb{F}_{q_0^{}}$ as a vector space over its subfield $\mathbb{F}_p$. Using Lemma~\ref{Fields1} we prove that $$1,\alpha$$ form a basis of $\mathbb{F}_{q}$ over its subfield $\mathbb{F}_{q_0^{}}$. Thus, the elements $$1,\,\alpha,\,{\beta=\alpha^2},\,{\alpha\beta=\alpha^3},\,{\beta^2=\alpha^4},\,{\alpha\beta^2=\alpha^5},\,\dots, \, {\beta^{2^{k-1}-1}=\alpha^{2^{k}-2}},\,{\alpha\beta^{2^{k-1}-1}=\alpha^{2^{k}-1}}$$ form a basis of $\mathbb{F}_{q}$ over its subfield  $\mathbb{F}_{p}$.

\end{proof}

\begin{lem}\label{Char2} {\rm  \cite[Lemma~11]{KMR_1}} Let $G$ be a simple group of Lie type over a field of characteristic $2$.
Then every subgroup of odd index is pronormal in $G$.

\end{lem}

\begin{lem}\label{MaxOddInSubgrE6} Let $G=E^{\varepsilon}_6(q)$, where $\varepsilon\in\{+,-\}$ and $q$ is odd, and let
$S$ be a Sylow $2$-subgroup of $G$. Then the following statements hold{\rm:}

$(1)$ $Z(S)$ is a cyclic subgroup.

$(2)$ If $t$ is the unique involution from $Z(S)$ and $C=C_G(t)$, then
$$L:=E(C)\cong Spin^{\varepsilon}_{10}(q),$$ $$Z:=C_G(L)\cong{\mathbb Z}_{(q- \varepsilon1)/(3,q- \varepsilon1),} \mbox{ and }$$
$$C/Z\cong {\rm Inndiag}(L/Z(L))\cong P\Omega^{\varepsilon}_{10}(q).{\mathbb Z}_{(4,q- \varepsilon1)}.$$

$(3)$ $N_G(S)=S\times R$, where $R=O(C)\cong {\mathbb Z}_{(q- \varepsilon1)_{2'}/(3,q- \varepsilon1)}$.

$(4)$ If $\varepsilon=-$, then $C$ is a maximal subgroup of $G$.

$(5)$ If $\varepsilon=+$, then $C$ is contained exactly in two parabolic maximal subgroups $P_1$ and $P_2$
of $G${\rm;} moreover, for each $i \in \{1,2\}$, the unipotent radical of $P_i$ is an elementary abelian subgroup of order $q^{16}$, $C$ is a Levi complement of $P_i$, and $P_1$ and $P_2$ are conjugate in $G$ by a graph automorphism of $G$ which normalizes $C$.

$(6)$ $N_G(S)$ is contained in a maximal subgroup $W$ of $G$ such that
$$E:=E(W)\cong Spin^+_8(q)\cong 2^2.P\Omega^+_8(q),$$
$$C_W(E)\cong (q- \varepsilon1)^2/(3,q- \varepsilon1),
\quad
Z(W)=1,$$
$$W/EC_W(E)\cong Sym_4,\quad\text{ and }\quad R<O(W).$$

$(7)$ If $q\equiv\varepsilon1 \pmod 4$, then $N_G(S)$ is contained in a maximal subgroup $N^{\varepsilon}$ of
$G$ and $$N^{\varepsilon}\cong (q- \varepsilon1)^6/(3,q- \varepsilon1). \Aut(U_4(2)).$$

$(8)$ If $q=p^m$, where $p$ is a prime, then for any odd prime divisor $r$ of $m$, $G$ has a field automorphism $\varphi$ of order $r$ such that $$G_0=G_0(r)=C_G(\varphi)$$ is maximal in $G$, $G_0$ contains a subgroup conjugate in $G$ to $S$, and
$$E(G_0)\cong E^{\varepsilon}_6(q_0),\quad \text{ where }\quad q_0^r=q,\quad\text{ and }\quad G_0\le {\rm Inndiag}(E(G_0)).$$

$(9)$ Each maximal subgroup of odd index of $G$ is conjugate in $G$ to one of the following subgroups{\rm:} $W$, $G_0(r)$,
$P_1$ or $P_2$ if $\varepsilon=+$, $C$ if  $\varepsilon=-$, $N^{\varepsilon}$ if $q\equiv\varepsilon 1 \pmod 4$.

\begin{proof} Proof of this follows directly from \cite[Theorens~5--7, Lemma~1.3]{KM}, \cite[Theorem, Paragraph~2]{LS}, \cite[Theorem,
Tables~5.1 and~5.2]{LSS}, \cite[Theorem~4.5.1, Proposition~4.9.1]{GLS}.

\end{proof}

\end{lem}

\begin{lem}\label{TorusNormalizer} Suppose that  $q\equiv\varepsilon1 \pmod 4$   and $M=N^{\varepsilon}$ is a maximal subgroup of odd index in $G=E_6^\varepsilon(q)$ which is the normalizer in $G$ of a maximal torus
$$T=T^\varepsilon\cong (q- \varepsilon1)^6/(3,q- \varepsilon1)$$ from Statement $(7)$ of Lemma~{\rm \ref{MaxOddInSubgrE6}}.

Let $\overline{\phantom{x}}:M\rightarrow M/T$ be the natural epimorphism, $S$ be a Sylow $2$-subgroup of $M$ {\rm(}and of~$G${\rm)}, and $t$ be a unique involution from $Z(S)$. Put $$V=\Omega_1(O_2(T)).$$
Then the following statements hold{\rm:}

$(1)$ $\overline{M}\cong \Aut(U_4(2))\cong GO_6^-(2)$.

$(2)$ $C_M(V)=T$ and $\overline{M}$ acts faithfully on $V$.

$(3)$ $V$ is an absolute irreducible $\mathbb{F}_2\overline{M}$-module isomorphic to the natural $\mathbb{F}_2 GO_6^-(2)$-module of dimension~$6$.

$(4)$ Maximal subgroups of $\overline M$ containing $\overline S$ are exhausted by subgroups $\overline{Q_1}$ and $\overline{Q_2}$, where $\overline{Q_1}$ is the stabilizer in $\overline M$ of a totally singular subspace $\langle t\rangle$ of dimension~$1$ from $V$,   $\overline{Q_2}$ is the stabilizer in $\overline M$ of a totally singular subspace $Y$ of dimension $2$ from $V$, $t\in Y$, $$\overline{Q_1}\cong 2^4:Sym_5, \mbox{ and } \mbox{ } \overline{Q_2}\cong (SL_2(3)\circ SL_2(3)).2^2.$$

$(5)$ If $r$ is a prime divisor of $|T|$ and $r>3$, then the subgroup $$T_r=\Omega_1(O_r(T))$$ is an elementary abelian $r$-group of rank~$6$, and $\overline{M}$ acts faithfully and absolutely irreducible on $T_r$ {\rm(}as on a vector space{\rm)}.

\end{lem}

\begin{proof} Statement $(1)$ follows directly from Statement~$(7)$ of Lemma~\ref{MaxOddInSubgrE6} and the following isomorphisms $\overline{M}\cong \Aut(U_4(2))\cong GO_6^-(2)$ (see~\cite{Atlas}).

We now prove Statement~$(2)$. Note that $V$ is a characteristic subgroup of $T$. Therefore,  $V\unlhd M$ and $C_M(V)\unlhd M$.
In particular, $V\unlhd S$, therefore a unique involution $t$ from $Z(S)$ belongs to $V$.

Let us prove that $C_M(V)=T$. The inclusion $T\le C_M(V)$ follows from the fact that $T$ is abelian. Now we have that $$\overline{C_M(V)}\unlhd\overline{M}\cong \Aut(U_4(2)).$$ Therefore, if $C_M(V)>T$, then the structure of the group $ \Aut(U_4(2))$
implies that $$|M:C_M(V)|\leq 2 \text{ and }M=C_M(V)S\le C_G(t).$$ This contradicts the maximality of $M$ and Lemma~\ref{MaxOddInSubgrE6}.
Thus, $\overline M$ acts faithfully on $V$.

Now consider Statement~$(3)$. Let $k$ be the algebraic closure of the field $\mathbb{F}_2$. Each composition factor of a faithful ${k}\overline{M}$-module $k\otimes V$ has dimension at most $\dim V=6$. By \cite{AtlasBr}, the group
$\overline{M}\cong\Aut(U_4(2))$ has (up to equivalence) a unique faithful irreducible representation over $k$ whose dimension is at most~$6$. Moreover, this dimension is equal to~$6$ and the natural representation of $GO_6^-(2)$ is absolutely irreducible.
Therefore, $\overline{M}$ acts irreducible on $V$, and the group $V$ as a $\mathbb{F}_2\overline M$-module is isomorphic to the natural $\mathbb{F}_2GO^-_6(2)$-module of dimension~$6$.

We can assume that $V$ is a space of dimension~$6$ over the field of order~$2$ with a non-degenerate quadratic form $Q$ of type~``$-$'' and $\overline M$ is the group of all non-degenerate linear transformations of $V$ stabilizing $Q$. This fact and the list of all maximal subgrous of the group $GO_6^-(2)$ (see, for example, \cite{Atlas}) imply Statement~$(4)$.

Let us prove Statement~$(5)$. The structure of $T$ implies that $T_r$ is an elementary abelian $r$-group of rank $6$ and $C_M(T_r)$ is a normal subgroup of $M$ containing $T$. Suppose that $T<C_M(T_r)$. Then, as in the proof of Statement~(2), we have $|M:C_M(T_r)|\leq 2$ and $M=C_M(T_r)S$ for any Sylow $2$-subgroup $S$ of $M$. By Statement~(7) of Lemma~\ref{MaxOddInSubgrE6}, we can assume that $N_G(S)<M$ and $N_G(S)=S\times R$, where $R=O(C_G(\Omega_1(S)))\cong {\mathbb Z}_{(q- \varepsilon1)_{2'}/(3,q- \varepsilon1)}$. Since
$N_{\overline{M}}(\overline{S})=\overline{S}$ (see \cite{Atlas}), $R$ is contained in $T$ and, therefore,
$R\cap T_r=\Omega_1(O_r(R))\cong {\mathbb Z}_r$. But then the normalizer $N_G(R\cap T_r)$ contains both subgroups $S$ and $C_M(T_r)$. Therefore, $N_G(R\cap T_r)$ contains  $M=C_M(T_r)S$. On the other hand, Lemma~\ref{MaxOddInSubgrE6}(2,3) implies that $N_G(R\cap T_r)=N_G(\Omega_1(O_r(R)))$ contains the subgroup  $C=C_G(t)$. A contradiction to Lemma~\ref{MaxOddInSubgrE6} because the maximality of subgroups $M$ and $C$ in $G$ implies that the subgroup $R\cap T_r$ is normal in the simple group~$G$. Thus, $\overline{M}$ acts faithfully on $T_r$.

Tables of ordinary (see \cite{Atlas}) and $5$-modular Brauer (see \cite{AtlasBr}) irreducible characters of the group
$\Aut(U_4(2))$  imply that the group $\overline{M}'$ has up to equivalence a unique  faithful irreducible representation of degree at most $6$ over the algebraic closure of the field $\mathbb{F}_r$, and its dimension is exactly~$6$. As in the case of characteristic~$2$ (see above), this fact implies that $\overline{M}$ acts irreducible on~$T_r$.

\end{proof}

\section{Examples of non-pronormal subgroups of odd index in groups $E^{\pm}_6(q)$}

\begin{prop}\label{Counterexample9dividesq-1} Let $G=E^{\varepsilon}_6(q)$, where $\varepsilon\in\{+,-\}$ and
$q$ is odd. If $9$ divides $q-\varepsilon1$, then $G$  contains a non-pronormal subgroup of odd index.
\end{prop}

\begin{proof} We use notation from Lemma~\ref{MaxOddInSubgrE6}. Consider a maximal subgroup $W$ of $G$ from Statement~$(6)$ of Lemma~\ref{MaxOddInSubgrE6} whose index in $G$ is odd.
Let $P$ be a Sylow $3$-subgroup of $O(W)$. Assume that $9$ divides $q-\varepsilon1$. Then we have
\begin{equation}\label{decomp}
  P\cong{\mathbb Z}_{3^n}\times{\mathbb Z}_{3^{n-1}}
\end{equation}
for some integer $n>1$. Put
$$
V=\Omega_1(P).
$$
Thus, $V$ is an elementary abelian $3$-group of rank~$2$.

We prove that $W$ contains a subgroup $H$ of odd index such that $H$ is not pronormal in~$HV$. It is easy to see that in this case $H$ is a subgroup of odd index in~$G$ which is not pronormal in~$G$.

 By Lemma~\ref{MaxOddInSubgrE6}, we have $$W/EC_W(E)\cong Sym_4.$$ Let $U$ be the complete preimage of $O_2(W/EC_W(E))$ in $W$ with respect to the natural homomorphism $W\to W/EC_W(E)$. Then we have $U\unlhd W$ and $W/U\cong Sym_3$. Let us fix a Sylow $2$-subgroup $S$ of $W$. Let $T=U\cap S$ and $N=N_W(T)$. It is clear that $T$ is a Sylow $2$-subgroup of $U$. Therefore, using the Frattini argument we prove that $W=NU$. It follows that $S< N$ and
$$
N/(N\cap U)\cong NU/U=W/U\cong Sym_3.
$$
This implies that there exists a $3$-element $x$ from $N$ whose image in $N/(N\cap U)$ generates $O_3(N/(N\cap U))$.

Put
$$
H=\langle S,x\rangle.
$$

Since $S\le H$, we see that $H$ is a subgroup of odd index both in~$W$ and in~$G$. We now prove that $H$ is not pronormal in~$HV$.
More precisely, we are going to prove that
\begin{itemize}
  \item $[H,V]$ coincides with $\mho^{n-1}(P)$ (note that $|\mho^{n-1}(P)|=3$), therefore,  $[H,V]<V$,

and
  \item $N_V(H)\le [H,V]$.
\end{itemize}

\noindent Then Lemma~\ref{CrterionOfProSub} implies that $H$ is not pronormal in $HV$.

 By Lemma~\ref{MaxOddInSubgrE6}, the subgroup $W$ contains $N_G(S)=S\times R$, where $R$ is a cyclic subgroup from $O(W)\le Z(EC_W(E))$ whose order is divisible by $3$ (because $9$ divides $q-\varepsilon1$). In particular, $1<V\cap R<V$  and $|V\cap R|=3$. Further, $$C_V(H)\le C_V(S)=V\cap R.$$ We now show that $C_V(H)=1$. Observe that $$EC_W(E)\le U\le  EC_W(E)S,$$ implies $R\le Z(U)$.  Moreover, $W=HU$ by the definition of $H$. If $C_V(H)\ne 1$, then $C_V(H)=V\cap R$ and $$1<V\cap R\le Z(W)=1,$$ see Lemma~\ref{MaxOddInSubgrE6}, a contradiction.

We claim  that $U= C_W(V)$. Indeed, $C_W(V\cap R)$ contains the subgroup $SEC_W(E)$ which has index $3$ in $W$. Taking into account that $Z(W)=1$, we have $C_W(V\cap R)= SEC_W(E)$. This implies that $$C_W(V)\le C_W(V\cap R)= SEC_W(E).$$ Moreover, $C_W(V)\unlhd W$ and we have that $C_W(V)/EC_W(E)$ is a normal $2$-subgroup in $W/EC_W(E)\cong Sym_4$. Therefore, either $C_W(V)=U$ or $C_W(V)=EC_W(E)$.
Since $W/C_W(V)$ is isomorphic to a subgroup of $\Aut(V)\cong GL_2(3)$, and $GL_2(3)$ not contain a subgroup isomorphic to $Sym_4$,  we conclude that $C_W(V)=U$.

Lemma~\ref{p'AutOfAbelianp}  implies that $C_W(P)=U$.

Let $\overline{\phantom{x}}:W\rightarrow W/U$ be the canonical epimorphism. Recall that $H$ acts on $V$ and on $P$ by conjugation.
These actions induce faithful actions of $\overline{H}=\overline{W}\cong Sym_3$ on $V$ and on $P$, respectively.
Let $I=\{\overline{t}_i\mid i=1,2,3\}$ be the set of all the involution from $\overline{H}$. Assume without loss of generality that $\overline{S}=\langle \overline{t}_1\rangle$.  Moreover, the group $\langle\overline{x}\rangle$ acts transitively on $I$ by conjugation.  The group $\overline{H}$ is generated by any two distinct involutions from~$I$.

Using Lemma~\ref{p'AutOfAbelianp}, we conclude that
$$
P=C_P(\overline{t}_1)\times [P,\overline{t}_1].
$$

 The group $C_P(\overline{t}_1)=C_P(S)$ coincides with a Sylow $3$-subgroup of $R$. Therefore, $C_P(\overline{t}_1)=C_P(S)$ is the cyclic group of order $3^{n-1}$.  This, and the decomposition (\ref{decomp}), imply that $$[P,\overline{t}_1]\cong \mathbb{Z}_{3^n}.$$ Therefore, $[P,\overline{t}_1]\cap \mho^{n-1}(P)\ne 1$ and, since $|\mho^{n-1}(P)|=3$, we have
$$
\mho^{n-1}(P)\le [P,\overline{t}_1].
$$

Further,
$\mho^{n-1}(P)\le V$, and the involution $\overline{t}_1$ inverts each element from $[P,\overline{t}_1]$. Therefore,
$
\mho^{n-1}(P)\le [V,\overline{t}_1].
$

It follows from $$C_V(\overline{t}_1)=C_V(S)=V\cap R\ne 1$$  that $
\mho^{n-1}(P)= [V,\overline{t}_1].
$
Finally, since $\mho^{n-1}(P)$ is a characteristic subgroup of~$W$ and  $\overline{H}=\overline{W}$ acts transitively on the set~$I$, we have $$
\mho^{n-1}(P)= [V,\overline{t}_i]
$$
for each $i \in \{1,2,3\}$. Since $V$ is abelian and $\overline{H}=\langle \overline{t}_i\mid i=1,2,3\rangle$, we conclude that
$$
[H,V]=[V,\overline{H}]=\prod_{i=1}^3 [V,\overline{t}_i]= \mho^{n-1}(P),
$$
as claimed.

Now we prove the containment $N_V(H)\le [H,V]$.

Note that for each $i \in \{1,2,3\}$ we have $[H,V]=[V,\overline{t}_i]$ and $$V=C_V(\overline{t}_i)\times [V,\overline{t}_i]=C_V(\overline{t}_i)\times [H,V].$$
Moreover, $$C_V(\overline{t}_i)\cap C_V(\overline{t}_j)=C_V(\langle\overline{t}_i,\overline{t}_j\rangle)=C_V(\overline{H})=1$$ for $1\le i<j\le 3$, and the number of subgroups of order $3$ in $V$, distinct from $[H,V]$, equals to $(9-3)/2=3$. Therefore, each of these subgroups coincides with  $C_V(\overline{t}_i)$ for some $i$, and the group $\overline{H}$ acts on these subgroups transitively.

Suppose to the contrary that $N_V(H)\nleq [H,V]$. Then  $N_V(H)$ contains $C_V(\overline{t}_i)$ for some $i$. However, $N_V(H)$ is invariant under $\overline{H}$, and therefore contains $C_V(\overline{t}_i)$ for each $i\in \{1,2,3\}$. 

 Now the subgroup $C_V(\overline{t}_1)=V\cap R$ of order $3$ acts on the set of all Sylow $2$-subgroups of $H$, and therefore acts on the set~$I$. As $C_V(\overline{t}_1)$ has a fixed point $\overline{t}_1$ on~$I$ and $|I|=3$, the action of $C_V(\overline{t}_1)$ on~$I$ is trivial. In particular, $C_V(\overline{t}_1)
\le C_V(\overline{t}_2)$. This contradicts the facts that $C_V(\overline{t}_i)\cap C_V(\overline{t}_j)=1$ if $i\ne j$ and $|C_V(\overline{t}_1)|=3$.

\end{proof}

\begin{prop}\label{CounterexampleInParabolic Case}
Let $q_0$ be a power of an odd prime $p$ and $q=q_0^r$ for an odd prime~$r$. Then the group $G=E_6^+(q)$
contains a non-pronormal subgroup of odd index.
\end{prop}

\begin{proof} We use notation from Lemma~\ref{MaxOddInSubgrE6}.

Let $\varphi$ be the canonical field automorphism of order $r$ of $G$ and $G_0=C_G(\varphi)$. Then by Lemma~\ref{MaxOddInSubgrE6} and \cite[Proposition~4.9.1]{GLS}, we have

\begin{itemize}
  \item[$(1)$] $|G:G_0|$ is odd{\rm;}
  \item[$(2)$]  $G_0'=E(G_0)\cong E_6(q_0)$ and $G_0$ is isomorphic to a subgroup of ${\rm Inndiag}(E_6(q_0))${\rm;} and
  \item[$(3)$] $G_0$ considered as a Chevalley group contains a parabolic subgroup $P_0$ of odd index such that $P_0$ is contained in a parabolic subgroup $P\in\{P_1, P_2\}$ of $G$, where $P_1$ and $P_2$ are subgroups from Lemma~\ref{MaxOddInSubgrE6}, and $P$ is invariant under~$\varphi$. Moreover, $P_0=C_P(\varphi)$ and the unipotent radical $O_p(P_0)$ of $P_0$ is contained in $O_p(P)$.
\end{itemize}

Let $S$ be a Sylow $2$-subgroup of $P_0$ and $V_0=O_p(P_0)$. Let us put $H=SV_0$  and prove that $H$ is not pronormal in~$P$.
More precisely, if $V=O_p(P)$ is the unipotent radical of~$P$, $C$ is a Levi complement in~$P$ such that $S\le C$, and $Z=Z(C)$, then $H$ is not pronormal in $ZSV$. Indeed, by Lemma~\ref{MaxOddInSubgrE6}, the subgroup~$Z$ is isomorphic to a subgroup of the multiplicative group~$\mathbb{F}_q^*$ and $|\mathbb{F}_q^*:Z|$ divides $3$, by the Zsigmondy theorem~\cite{z}, $Z$ contains an element whose order does not divide $|\mathbb{F}_{q^{}_0}^*|=q^{}_0-1$ (otherwise
$$q_0^r-1=|\mathbb{F}_{q}^*|\text{ divides }3|\mathbb{F}_{q^{}_0}^*|=3(q^{}_0-1),$$ whence  $$q_0^{r-1}+q_0^{r-2}+\dots+q^{}_0+1\text{ divides }3$$
and $q_0=2$, a contradiction to the fact that $p$ is odd).

Let $g\in Z$ be an element whose order does not divide $q^{}_0-1$. Since $Z$ is contained in a Cartan subgroup of~$G$, the element~$g$ corresponds to some character
$$\chi:\mathbb{Z}\Phi\rightarrow \mathbb{F}_q^*, \,\text{ where }\,\Phi\text{ is a root system of type }E_6,$$
and each root subgroup $X_r=\{x_r(\alpha)\mid \alpha\in \mathbb{F}_q\}$ for $r\in\Phi$ is invariant under $g$ with respect to the action $$g^{-1}x_r(\alpha)g=x_r(\chi(r)^{-1}\alpha)$$ (see \cite[7.1, in particular, p.~100]{Car}).
The element $g$ was chosen such a way that $\chi(r)\notin \mathbb{F}_{q^{}_0}^*$ for some root $r\in\Phi$, where
$X_r\le V$ since $g$ belongs to the center of a Levi complement~$C$. Further,
$$V_0= \langle x_s(\alpha)\mid X_s\le V, \alpha\in \mathbb{F}_{q^{}_0}\rangle,$$
and there exists an ordering of roots such that each element from $V$ and each element from $V_0$ has a unique representation as a product of root elements $x_s(\alpha)$  corresponding to distinct roots~$s$ taken with respect to this ordering.  Thus, $$g^{-1}x_r(1)g=x_r(\chi(r)^{-1})\in V_0^g\setminus V_0$$
and the subgroups $V_0$ and $V_0^g$ are distinct.  But $g\in Z(C)$, therefore, $S^g=S$. Thus,
$$
\langle H, H^g\rangle= \langle V_0, V_0^g, S\rangle.
$$
Suppose that the subgroups $H$ and $H^g$ are conjugate in the subgroup $\langle H, H^g\rangle$. Then the subgroups $V_0=O_p(H)$ and
$V_0^g=O_p(H^g)$ are conjugate in $\langle H, H^g\rangle$. However, the subgroups $V_0$ and $V_0^g$ are normal in $\langle H, H^g\rangle= \langle V_0, V_0^g, S\rangle$ (since both $V_0$ and $V_0^g$ are invariant under $S$ and centralize each other because they both are subgroups of $V$ which is abelian), therefore,  $V_0$ and $V_0^g$ are not conjugate in $\langle H, H^g\rangle$. This is a contradiction.
\end{proof}

\section{Proof of Theorem}

Let $G=E^{\varepsilon}_6(q)$, where $\varepsilon\in\{+,-\}$, $q=p^m$, and $p$ is a prime.

Propositions~\ref{Counterexample9dividesq-1} and~\ref{CounterexampleInParabolic Case} imply that if all the subgroups of odd index are pronormal in $G$, then $q\not\equiv \varepsilon1\pmod {18}$ and if $\varepsilon= +$, then $m$ is a power of $2$.

Let us prove the converse.
Assume that $q\not\equiv \varepsilon1\pmod {18}$ and if $\varepsilon= +$, then $m$ is a power of $2$. Now we argue that all the subgroups of odd index are pronormal in $G$.

We can assume that $q$ is odd by Lemma~\ref{Char2}. Let $H$ be a subgroup of odd index in $G$, and let  $S$ be a Sylow $2$-subgroup of $H$ and $G$, and let $g$ be an element of odd order from $N_G(S)$. By Lemma~\ref{MaxOddInSubgrE6}, we have $N_G(S)=S\times R$, where $R$ is the cyclic group of order $(q-\varepsilon1)_{2'}/(3,q-\varepsilon1)$. Therefore $g\in R\le Z(N_G(S))$ and $|g|$ is not divisible by $3$.
To prove that $H$ is pronormal in $G$ it is
sufficient to prove by Lemma~\ref{NormSyl} that $H$ and $H^g$ are conjugate in $K=\langle H, H^g\rangle$. There is nothing to prove if $K=G$, so we can assume that $K<G$. Thus, there exists a maximal subgroup $M$ of $G$ such that $K\le M$. It is easy to see that $M$ is one of the maximal subgroups of odd index in $G$  listed in Lemma~\ref{MaxOddInSubgrE6}.

Observe that we can assume that $g\in M$. Indeed, Lemma~\ref{MaxOddInSubgrE6} implies that $N_G(S)\le M$ or $M=C_G(\varphi)$ for a field automorphism $\varphi$ of odd prime order~$r$. In the latter, let us suppose that $g\notin C_G(\varphi)$. We have
  $$
  x:=[g^{-1},\varphi]=g(g^{-1})^\varphi\ne 1.
  $$
  Note that $x$ centralizes $H$. Indeed, if $h\in H$, then since
$$h,h^g\in  K\le M=C_G(\varphi)=C_G(\varphi^{-1}),$$
we have
  $$h^x=(h^g)^{\varphi^{-1} g^{-1}\varphi}=(h^g)^{ g^{-1}\varphi}=h^{\varphi}=h.$$
Since $S\le H$, we have that $x\in N_G(S)$. But $g\in Z(N_G(S))$, therefore,
$$g\in C_G(x)<G.$$
Let $M_0$ be a maximal subgroup of~$G$ which contains $C_G(x)$. Since $H\le C_G(x)$, we have that $M_0$ is a subgroup of odd index in $G$ containing $H$, $g$, and $K=\langle H, H^g\rangle$, as claimed.

Thus, by Lemma~\ref{Overgroup}, to prove that all the subgroups of odd index are pronormal in $G$ it is enough to prove that all the subgroups of odd index are pronormal in any maximal subgroup $M$ of odd index of $G$. Let us consider these subgroups case by case with respect to Lemma~\ref{MaxOddInSubgrE6}. So, below we use notation from Lemma~\ref{MaxOddInSubgrE6}.

\medskip

C~a~s~e~~1: $M=G_0=C_G(\varphi)$, where $\varphi$ is a field automorphism of odd prime order~$r$ of the group~$G$.

It is known that any two elements of order $r$ from the coset $G\varphi$ of the group $Aut(G)$ by $Inn(G)$ (which is identified to $G$), are conjugate by an inner-diagonal automorphism сопряжены  (see \cite[Proposition~4.9.1]{GLS}). Therefore centralizers of these elements in $G$ are isomorphic. Let us identify $\varphi$ to an automorphism $\overline{\phantom{x}}:\mathbb{F}_q\rightarrow \mathbb{F}_q$ of the field $\mathbb{F}_q$ such that for each root element $x_r(\alpha)$ of a non-twisted group of Lie type $E_6$ (either this group coincides with $G$ or $G$ was constructed via this group), the following equalities hold
  $$x_r(\alpha)^\varphi= x_r(\overline{\alpha})=x_r({\alpha}^\varphi)$$
(see \cite[12.2, in particular, p.~200]{Car}). Let $\mathbb{F}_{q_0^{}}$ be the subfield of $\mathbb{F}_q$ which consists of fixed points of~$\varphi$, thus, $q=q_0^r$. Let us prove that $G_0\cong E_6^\varepsilon(q_0)$ and then use inductive reasonings to prove that all the subgroups of odd index are pronormal in $G_0$. It is enough to prove that $|G_0/E(G_0)|=1$.

Suppose for the contradiction that  $|G_0/E(G_0)|\not=1$. By~\cite[Proposition~4.9.1]{GLS}, the group $E(G_0)\cong E_6^\varepsilon(q_0^{})$ has an outer diagonal automorphism, therefore,
$$q_0^{}\equiv\varepsilon1\pmod 3.$$
Recall that  $9$ does not divide the number $$q-\varepsilon1=(q_0-\varepsilon1)\left(\displaystyle\sum_{i=0}^{r-1}(\varepsilon1)^iq_0^{r-i-1}\right),$$
therefore a Sylow $3$-subgroup of the multiplicative group of the field $\mathbb{F}_q$ is contained in the subfield
$\mathbb{F}_{q_0^{}}$ which consists of fixed points of~$\varphi$. Consider an extended Cartan subgroup $\hat H$ of $G$,
which is invariant under~$\varphi$, and let $\hat G=\hat{H}G$ be the group of inner-diagonal automorphisms of~$G$. Since $|\hat{ G}:G|=3$, there exists a $3$-element $\delta$ from $\hat{H}$ such that $\hat G=\langle G,\delta\rangle$. Since $\varphi$ centralizes a Sylow $3$-subgroup of the multiplicative group of the field $\mathbb{F}_q$, we have that $\varphi$ considered as a field automorphism of~$G$ and of $\hat G$ centralizes a Sylow $3$-subgroup of an abelian group $\hat H$. In particular, $\varphi$ centralizes~$\delta$.

It is known (see \cite[Proposition~4.9.1]{GLS}) that $C_{\hat G}(\varphi)$ is isomorphic to the group of inner-diagonal automorphisms of $G_0'\cong E_6^\varepsilon(q^{}_0)$. Since $\delta\in C_{\hat G}(\varphi)$, we have
$\hat G=GC_{\hat G}(\varphi)$, therefore,
$$C_{\hat G}(\varphi)/C_{G}(\varphi)\cong \hat G/G\cong \mathbb Z_3.$$
Thus, $G_0=C_{G}(\varphi)$ coincides with a unique normal subgroup of index $3$ of
$C_{\hat G}(\varphi)$, therefore, $G_0\cong E_6^\varepsilon(q^{}_0).$

\medskip

C~a~s~e~~2:  $M=C=C_G(t)$ is the centralizer in $G$ of an involution~$t$ from the center of a Sylow $2$-subgroup $S$ of~$G$.

Note that $C$ is a maximal subgroup of $G$ only if $\varepsilon=-$. We omit this condition in our reasonings because if $\varepsilon=+$, then we need to prove that the subgroups of odd index are pronormal in $C=C_G(t)$ to consider cases~4 and~5 below.

If $H$ is a subgroup of~$C$ containing~$S$ and~$g$ is an element of odd order from $N_G(S)$, then $g\in Z(C)$ by Lemma~\ref{MaxOddInSubgrE6} and, therefore, $H=H^g$. Thus,  $H$ is pronormal in $C$ by Lemma~\ref{NormSyl}.

\medskip

C~a~s~e~~3: $M=W.$

Recall that $g\in R$ and $R$ is a $\{2,3\}'$-subgroup of $O(W)$, where $O(W)$ is a normal abelian subgroup of~$W$ which is contained in~$C_W(E)$. Thus, in this case it is sufficient to prove that $H$ is pronormal in $O_{\{2,3\}'}(W)H$.
Let us prove that for each $H$-invariant subgroup $U$ of $O_{\{2,3\}'}(W)$ the inclusion (really, the equality)
$U\le N_U(H)\,[H,U]$ holds, then the pronormality of $H$ in $O_{\{2,3\}'}(W)H$ will follow from Lemma~\ref{CrterionOfProSub}. Since $U$ is a direct product of its Sylow subgroups, and each of these Sylow subgroups is $H$-invariant, it is sufficient to consider a case when $U$ is an $r$-group for some prime~$r>3$.

By Lemma~\ref{MaxOddInSubgrE6}, $O(W)$ is contained in the center of  $EC_W(E)$ and $$W/ EC_W(E)\cong Sym_4.$$ So that, $U$ is contained in the center of  $EC_W(E)$.
It follows that $\overline{H}:=H/C_H(U)$ is a $\{2,3\}$-group. By Lemma~\ref{p'AutOfAbelianp}, we have
$$
U=C_U(\overline{H})\,[\overline{H},U]=C_U({H})\,[{H},U]\le N_U(H)\,[H,U].
$$

\medskip

C~a~s~e~~4: $M=N^{\varepsilon}$, where $q\equiv\varepsilon1\pmod 4$.

By Lemma~\ref{MaxOddInSubgrE6}, $M$ contains an abelian normal subgroup
$T$ of the form $(q- \varepsilon1)^6/(3,q- \varepsilon1)$ such that $M/T$ is isomorphic to
$\Aut(U_4(2))\cong GO^-_6(2)$. Let $\overline{\phantom{x}}$ be the canonical epimorphism $M\rightarrow M/T$.
It is known  (see \cite{Atlas}) that a Sylow $2$-subgroup of the group $\overline{M}$ coincides with its normalizer in $\overline{M}$.
Since $q\not\equiv \varepsilon1\pmod {18}$, we have that $R$ is contained in $O_{\{2,3\}'}(T)$. Thus, by Lemmas~\ref{NormSyl}, \ref{HV=>VB}, and~\ref{CrterionOfProSub}, it is sufficient to prove that $H$ is pronormal in $HO_{\{2,3\}'}(T)$.

Let $S_0$ be a Sylow $2$-subgroup of $T$. Put $V=\Omega_1(S_0)$. By Lemma~\ref{TorusNormalizer}, we have that $V$ is an elementary abelian $2$-group of rank~$6$, and a unique involution~$t$ from $Z(S)$ belongs to~$V$. Lemma~\ref{TorusNormalizer} implies that one of the following cases holds{\rm:}

 \begin{itemize}
   \item[$(4a)$] $\overline H\leq \overline{Q_1}$,  where $\overline{Q_1}$ is the stabilizer in~$\overline M$ of an isotropic subspace $\langle t\rangle$ of dimension~$1$ from $V$, and $$\overline{Q_1}\cong 2^4:Sym_5;$$
   \item[$(4b)$] $\overline H\leq \overline{Q_2}$, where  $\overline{Q_2}$ is the stabilizer in~$\overline M$ of an isotropic subspace $Y$ of dimension~$2$ from~$V$, and
$$\overline{Q_2}\cong (SL_2(3)\circ SL_2(3)).2^2;$$
   \item[$(4c)$] $\overline H=\overline M$.
 \end{itemize}

If Case~$(4a)$ appears, then we reduce to Case~2 considered above.

If Case~$(4b)$ appears, then $\overline{H}$ is a $\{2,3\}$-group and, as in Case~3, we have $$
U=C_U(\overline{H})\,[\overline{H},U]=C_U({H})\,[{H},U]\le N_U(H)\,[H,U].
$$

Assume that Case~$(4c)$ occurs.
Now we prove that $U=N_U(H)[H,U]$ for each $H$-invariant subgroup~$U$ of~$O_{\{2,3\}'}(T)$ and use Lemma~\ref{CrterionOfProSub}. As in Case~3, we can assume that $U$ is an $r$-group for some prime~$r>3$ from~$O_{\{2,3\}'}(T)$, whence the equality $U=[{H},U]= N_U(H)\,[H,U]$ obviously follows.   Indeed, by Lemmas~\ref{MaxOddInSubgrE6} and~\ref{TorusNormalizer}, the following statements hold{\rm:} $C_{O(T)}(\overline{Q_1})=R$, $C_{O(T)}(\overline{Q_2})\leq R$, and
$\overline{M}=\langle\overline{Q_1},\overline{Q_2}\rangle$. But Lemma~\ref{TorusNormalizer} implies that the group $\overline{M}$
acts irreducible on $\Omega_1(O_r(T))$ for each prime divisor~$r$ of the number $|O_{\{2,3\}'}(T)|$, therefore,
$C_{O_r(T)}(\overline{Q_2})=1$ by Lemma~\ref{p'AutOfAbelianp}. Since $\overline{Q_2}$ is a $\{2,3\}$-group, we have $U=C_U(\overline{Q_2})\times [\overline{Q_2},U]=[\overline{Q_2},U]\leq [\overline{M},U]$ by Lemma~\ref{p'AutOfAbelianp}. Thus,
$U=[\overline{H},U]$.

\medskip

To complete the proof it remains to consider the following case.

C~a~s~e~~5: $\varepsilon=+$ and $M$ is conjugate to a subgroup $P\in\{P_1,P_2\}$ from Lemma~\ref{MaxOddInSubgrE6}.

We use Lemma~\ref{Important}. Let $V$ be the unipotent radical of~$P$.
Sylow $2$-subgroups of~$P/VZ\cong C/Z$ are self-normalized (see \cite{Kond1} and Lemma~\ref{N_G(S)=S}). A Sylow $2$-subgroup $T$ of $Z$ coincides with a Sylow $2$-subgroup of the group $VZ$, and $N_P(T)$ coincides with a Levi complement~$C$ of~$P$.
Using Case~2 considered above we conclude that all the subgroups of odd index are pronormal in~$C$. By Lemma~\ref{Important},
it is sufficient to prove that all the subgroups of odd index are pronormal in~$VZ$.

Let $H_1$ be a subgroup of odd index in~$VZ$. Consider a subgroup $U=H_1\cap V=O_p(H_1)$. The Schur--Zassenhaus theorem implies that $H_1=XU$, where $X$ is a Hall $p'$-subgroup of~$H_1$, and by the Hall theorem, we can assume that $X\le Z$ since $Z$ is a Hall $p'$-subgroup of~$VZ$ and $VZ$ is solvable. In particular, $X$ contains a Sylow $2$-subgroup $T$ of $Z$.
Let us prove that $U$ is normal in~$VZ$. It is easy to see that $U$ is normal in $V$ since $V$ is abelian. Thus, it is sufficient to prove that $U$ is $Z$-invariant, i.e. $u^z\in U$ for each $u\in U$ and each $z\in Z$.

By~\cite[Table~3]{Kor}, $C$ acts on $V$ by conjugation and induces on~$V$ a faithful irreducible $\mathbb{F}_q C$-module. The Clifford theorem~\cite[Theorem~3.4.1]{G} implies that $\mathbb{F}_q Z$-module $V$ is completely reducible and is equal to a direct product of irreducible $\mathbb{F}_q Z$-modules which are pairwise conjugate in $C$. By \cite[Theorem~3.2.4]{G}, since the order of the cyclic group $Z$ divides $q-1$, the group~$Z$ acts on $V$ with scalar action, i.\,e. for any element $z$ from $Z$ there exists an element $\beta\in \mathbb{F}^*_q$ such that $v^z=\beta v$ for each $v\in V$.

The subgroup $U$ defined above could be considered as a subspace of~$V$ considering as a vector space over the prime subfield $\mathbb{F}_p$ of the field~$\mathbb{F}_q$.

Recall that the degree of $\mathbb{F}_q$ over~$\mathbb{F}_p$ is $2^k$.
Let $x$ be a generator of a Sylow $2$-subgroup of $Z$ and $v^x=\alpha v$ for all $v\in V$. Since
$V$ is a faithful module and $|Z|=(q-1)/(3,q-1)$, the element $\alpha$ is a generator of a Sylow $2$-subgroup of the group $\mathbb{F}^*_q$, and by Lemma~\ref{Fields2}, elements $$1,\alpha,\alpha^2,\dots, \alpha^{2^k-1}$$ form a basis of the field $\mathbb{F}_q$ as a vector space over~$\mathbb{F}_p$.

Choose an arbitrary element $u\in U$. Since $x\in T\le H_1$, we have
$$\alpha^i u=u^{x^i}\in U.$$

If for an arbitrary element $z\in Z$  the equality $u^z=\beta u$ holds, where $\beta\in \mathbb{F}^*_q$, then for some  $\lambda_i\in\mathbb{F}_p$ we have
$$\beta=\sum \lambda_i\alpha^i\quad \text{ and }\quad u^z=\beta u=\sum \lambda_i(\alpha^i u)=\sum \lambda_iu^{x^i}\in U.$$

Take $g\in N_{VZ}(T)=Z$. Then $g$ normalizes $X$ since $Z$ is abelian, and $g$ normalizes $U$ by the fact proved above.
Then $g$ normalizes $H_1=XU$, and in view of Lemma~\ref{NormSyl}, the subgroup $H_1$ is pronormal in $VZ$. This completes the proof of the theorem.

\section{Acknowledgements}

This work was supported by the Russian Science Foundation (project 19-71-10067).

The authors are thankful to Prof. Stephen Glasby for his helpful comments which improved this text.

\bigskip

Anatoly~S. Kondrat'ev

Krasovskii Institute of Mathematics and Mechanics UB RAS, Yekaterinburg, Russia

E-mail address: a.s.kondratiev@imm.uran.ru

\medskip

Natalia~V. Maslova

Krasovskii Institute of Mathematics and Mechanics UB RAS, Yekaterinburg, Russia

Ural Federal University, Russia

E-mail address: butterson@mail.ru

ORCID: 0000-0001-6574-5335

\medskip

Danila~O. Revin

Krasovskii Institute of Mathematics and Mechanics UB RAS, Yekaterinburg, Russia

Sobolev Institute of Mathematics SB RAS, Novosibirsk, Russia

Novosibirsk State University, Russia

E-mail address: danila.revin@gmail.com

\end{document}